\documentclass[11pt, letterpaper]{amsart}
\usepackage[mathscr]{eucal}
\usepackage{palatino, mathpazo, amsfonts, mathrsfs, amscd}
\usepackage[all]{xy}
\usepackage{url}
\usepackage{amssymb, amsmath, amsthm}
\usepackage{stackrel}
\usepackage{bm}
\usepackage{stmaryrd}

\usepackage{tikz-cd}
\usepackage{tikz}
\usetikzlibrary{arrows}

\usepackage[
colorlinks=false,
linkcolor =black, 
citecolor = black,
]{hyperref}

\newtheorem{theorem}{Theorem}[section]
\newtheorem{lemma}[theorem]{Lemma}

\newtheorem{proposition}[theorem]{Proposition}
\newtheorem{corollary}[theorem]{Corollary}

\newtheorem{question}[theorem]{Question}
\newtheorem{construction}[theorem]{Construction}
\theoremstyle{remark}
\newtheorem{remark}[theorem]{Remark}
\newtheorem{example}[theorem]{Example}
\newtheorem{definition}[theorem]{Definition}

\newtheorem{notation}[theorem]{Notation}

\numberwithin{equation}{subsection}
\usepackage{todonotes}

\newcommand{\spec}{\operatorname{Spec}}

\newcommand{\cc}[1]{\mathcal{#1}}  % mathcal
\newcommand{\af}{\mathbb{A}}
\newcommand{\CC}{\mathbb{C}}
\newcommand{\ZZ}{\mathbb{Z}}

\newcommand{\GM}{\mathbb{G}_m}

\newcommand{\PP}{\mathbb{P}}
\newcommand{\QQ}{\mathbb{Q}}

\newcommand{\RR}{\mathbb{R}}

\DeclareMathOperator{\Hom}{Hom}

\DeclareMathOperator{\Amp}{Amp}

\DeclareMathOperator{\NE}{NE}

\DeclareMathOperator{\Sym}{Sym}

\newcommand{\op}[1]{\operatorname{#1}}

\title{A Kleiman criterion for GIT stack quotients}
\author[Shoemaker]{Mark Shoemaker}
\address{
  \begin{tabular}{l}
	   Mark Shoemaker \\
   \hspace{.1in} Colorado State University \\
      \hspace{.1in} Department of Mathematics \\
   \hspace{.1in} 1874 Campus Delivery  \\
   \hspace{.1in} Fort Collins, CO, USA, 80523-1874\\
   \hspace{.1in} Email: {\bf mark.shoemaker@colostate.edu} \\
  \end{tabular}
}

\begin{document}

\maketitle

\begin{center}
{\emph{In memory of Bumsig Kim}}\end{center}

\begin{abstract}Kleiman's criterion states that, for $X$ a projective scheme, a divisor $D$ is ample if and only if it pairs positively with every non-zero element of the closure of the cone of curves.  In other words, the cone of ample divisors in $N^1(X)$ is the interior of the nef cone.  In this paper we present an analogous statement for a variety $X$ acted on by a reductive group $G$ with a choice of $G$-linearization $L \to X$.  In this new context, the ample cone of $X$ is replaced by a cell in the variation of GIT decomposition of the G-ample cone, and curves in $X$ are replaced by quasimaps to $[X/G]$.

\end{abstract}
%\tableofcontents
\section{Introduction}

Let $X$ be a projective scheme over an algebraically closed field.  Let $N^1(X)_\RR$ denote the group of $\RR$-divisors modulo numerical equivalence and denote by 
$\Amp(X) \subset N^1(X)_\RR$ the convex cone spanned by ample divisors.  
%Let $\Nef(X)$ denote the nef cone of numerically effective divisor classes.  
Let $N_1(X)_\RR$ be the group of 1-cycles in $X$ up to numerical equivalence.  Define the cone of curves $\NE(X)$ to be the convex cone spanned by effective curve classes in $N_1(X)_\RR$. 

It is a remarkable fact that the ampleness of a divisor can be determined solely based on its intersection with 1-cycles.
%Kleiman's criterion is a useful test for ampleness.
\begin{theorem}[Kleiman's criterion \cite{Klei}]\label{tK}
A divisor $D$ on $X$ is ample if and only if 
$$D \cdot \gamma > 0$$
for all $\gamma \in \overline{\NE(X)} \setminus {\mathbf 0}$.  In other words, the ample cone is the interior of the dual of the cone of curves:
\begin{equation}\label{e-1}\Amp(X) = \op{int}\left(\NE(X)^\vee\right).\end{equation}
%and the
\end{theorem}
The goal of this paper is to obtain an analogous statement in the setting of variation of GIT.
\subsection{variation of GIT}
Let $G$ be a reductive group acting on $X$.  As is now well-understood \cite{GIT, DH, Thad}, there is not a canonical algebraic variety representing the ``quotient'' of $X$ by $G$; one must first fix a linearization, i.e., a $G$-equivariant line bundle $L \to X$.  Given such a choice of $L$, one can construct the \emph{GIT quotient} $X \sslash_L G$, which is a categorical quotient of the $L$-semi-stable locus $X^{ss}(L)$, an open $G$-invariant subset of $X$.  

Fix an ample $G$-equivariant line bundle $L \to X$ such that the semi-stable locus $X^{ss}(L)$ is nonempty.  We will think of the triple of data of $(X, G, L)$ as a \emph{$G$-linearized variety $X$}.  We would like to understand if there is an analogue of Theorem~\ref{tK} for $(X, G, L)$.

If two $G$-equivariant ample line bundles $L, L' \to X$ have the same semi-stable locus we say $L$ and $L'$ are GIT equivalent.  In this case, the GIT quotients $X\sslash_L G$ and $X \sslash_{L'} G$ are isomorphic.
%By \cite{DH}, there are  finitely many such GIT equivalence classes, and therefore finitely many GIT quotients of $X$ by $G$.
%
Denote by $\op{NS}^G(X)$ the group of $G$-equivariant line bundles up to $G$-algebraic equivalence (see Section~\ref{s:vGIT}) and let $\op{Amp}^G(X) \subset \op{NS}^G(X)_\QQ$ denote the cone generated by ample line bundles.
Define $C^\circ(L)$ to be the cone of classes of $G$-ample line bundles which are GIT equivalent to $L$:
$$C^\circ(L) = \{[L'] \in \op{Amp}^G(X) | X^{ss}(L) = X^{ss}(L')\}.$$
Whenever $[L'] \in C^\circ(L)$, the line bundle $L' \to X$ descends to an ample line bundle on the GIT quotient $X \sslash_L G$.  In fact, under sufficiently nice conditions the cone $C^\circ(L)$ may be identified with $\Amp(X \sslash_L G)$.
It is reasonable, therefore, to view $C^\circ(L)$ as the analogue of the ample cone for the triple $(X, G, L)$.  

On the other hand, in general the equivalence classes  $$\{C^\circ(L)\}_{L \in \op{Amp}^G(X)}$$ carry more information than just the isomorphism class of the GIT quotient $X \sslash_L G$, as illustrated by the following example.

\begin{example}\label{e:count} For simplicity we give a toric example.  
Let $T = \GM^2$ act on $V = \af^3$ with charge matrix
\[\left(\begin{array}{ccc}
2&1&0 \\
0&1&1
\end{array}\right).
\]
Consider the characters $\theta_+, \theta_- \in \chi(T)$ with the following weights
$$\theta_+ = \left(\begin{array}{c}
4 \\
2
\end{array}\right),\hspace{.5 cm} \theta_- =\left(\begin{array}{c}
2 \\
4
\end{array}\right).$$
The characters $\theta_+$ and $\theta_-$ define $T$-linearizations of the trivial line bundle $\cc O_V$, which we denote by  $\cc O_V(\theta_+)$ and $\cc O_V(\theta_-)$.
Although  $\theta_+$ and $\theta_-$ are not GIT equivalent, the respective GIT quotients $V \sslash_{\theta_+} T$ and $V \sslash_{\theta_-} T$ are both isomorphic to $\PP^1$.  Furthermore, the $T$-equivariant line bundles  $\cc O_V(\theta_+)$ and $\cc O_V(\theta_-)$ both 
descend to $\cc O_{\PP^1}(1)$ on the GIT quotient.

In fact, the GIT equivalence classes carry not just the information of the GIT quotient, they also distinguish the various \emph{GIT stack quotients}.  In this case, the GIT stack quotients $[V^{ss}(\theta_+) / T]$ and $[V^{ss}(\theta_-) / T]$ are not isomorphic.  Over $\CC$ the first is the ``football'' $[\PP^1_\CC/ \mu_2]$ while the second is the ``teardrop'' $\mathbb{W}\PP_\CC(2,1)$.  
%Although these stacks have isomorphic coarse spaces, they are distinct orbifolds.  
%
%Here the GIT equivalence classes record not just 
%
%  Then $X_+ = [\PP^1/ \mu_2]$ and $X_- = \PP(2,1)$.  Indeed $X_+$ has two points with nontrivial isotropy while $X_-$ has only one.
\end{example}

\subsection{Quasimaps}
If the GIT equivalence class $C^\circ(L)$ plays the role of the ample cone for the triple $(X, G, L)$, what is the correct analogue of the cone of curves?  From Example~\ref{e:count}, we see that it is not sufficient to simply consider the cone of curves of the GIT quotient,  $\NE(X \sslash_L G)$, because 1-cycles in $\PP^1 =  V\sslash_{\theta_+} T = V\sslash_{\theta_-} T$ cannot distinguish between $\cc O_V(\theta_+)$ and $\cc O_V(\theta_-)$.  

We propose that the correct replacement for curves in this context is given by \emph{$L$-quasimaps}.  An $L$-quasimap is a morphism $$f: C \to [X/G]$$ from a smooth curve $C$ to the stack quotient $[X/G]$ such that the preimage of the $L$-semi-stable locus $f^{-1}([X^{ss}(L)/G])$ is a dense open subset of $C$.

There is a notion of degree for a morphism $f: C \to [X/G]$, which naturally gives an element $ \op{deg}(f) \in \Hom(\op{NS}^G(X), \ZZ)$.  Define $$\op{NE}(L) \subset \Hom(\op{NS}^G(X), \QQ)$$ to be the cone generated by the degrees of $L$-quasimaps.  This plays the role of the cone of curves for $(X, G, L)$.

In Section~\ref{s:mc}, we show that a certain class of quasimaps is closely related to the Hilbert--Mumford criterion.  This is the main tool of the paper.

\subsection{Results}
Let $G$ be a reductive group acting on a normal projective variety $X$ over an algebraically closed field $k$.  Fix an ample $G$-equivariant line bundle $L \to X$ such that the semi-stable locus $X^{ss}(L)$ is nonempty.  Our main result is the following analogue of Kleiman's criterion.
\begin{theorem}[Theorem~\ref{maint}] The following cones are equal
\begin{equation}\label{e0}C^\circ(L) = \op{relint}\left(\op{NE}(L)^\vee\right)\cap \op{Amp}^G(X),\end{equation}
where $\op{relint}$ denotes the relative interior.  
\end{theorem}
In particular, one can test whether a $G$-ample line bundle $L' \to X$  is GIT equivalent to $L$ by intersecting it with classes in $\overline{\op{NE}(L)}$.
The relative interior is necessary, as the cone $C^\circ(L)$ may be contained in a proper subspace of $\op{NS}^G(X)_\QQ$. 

If we restrict our attention to linearizations of a fixed ample line bundle $P \to X$, the statement takes a form even closer to \eqref{e-1}.
Let $\op{NS}^G_{P}(X)$ denote the group of all linearizations of all powers of $P$.  
%Elements of $\Hom(\op{NS}^G(X), \QQ)$ naturally restrict to $\op{NS}^G_{P}(X)$, and we  define $\op{NE}_P(L)$ to be the image of $\op{NE}(L)$ under this restriction.
%Fix $L \to X$ a particular linearization of $P$, and let  $A_P(L) = C^\circ(L) \cap \op{NS}^G_{P}(X)$ be the cone of all  linearizations of powers of $P$ which are GIT equivalent to $L$.  In this case we prove the following.
\begin{theorem}[Theorem~\ref{fixedt}]\label{t14} Let $L \in \op{NS}^G_{P}(X)$ be a linearization of a positive power of $P$.  Let $A_P(L) = C^\circ(L) \cap \op{NS}^G_{P}(X)$ be the cone of all  linearizations of powers of $P$ which are GIT equivalent to $L$ and let $\op{NE}_P(L)$ be the cone of degrees of $L$-quasimaps on $\op{NS}^G_{P}(X)$.
If $X^{ss}(L)$ is nonempty, then $$A_P(L) = \op{relint}\left(\op{NE}_P(L)^\vee\right).$$
\end{theorem}
\begin{corollary}If $\op{NS}(X)_\QQ \cong \QQ$, then $$C^\circ(L) = \op{relint}\left(\op{NE}(L)^\vee\right).$$
\end{corollary}
Next, we look at what can be said beyond the ample cone.
If $\op{Pic}(X)_0$ is torsion, then $\op{NS}^G(X)_\QQ = \op{Pic}^G(X)_\QQ$.  In this case GIT equivalence classes are well defined on all of $\op{Pic}^G(X)_\QQ$ and not just in the ample cone $\op{Amp}^G(X)$.
Define 
$$A(L) := \{L' \in \op{Pic}^G(X)_\QQ | X^{ss}(L') = X^{ss}(L)\}$$
to be the GIT equivalence class of $L$ in $\op{Pic}^G(X)_\QQ$.

We obtain results in two different contexts.  The first is for generalized flag varieties. 
\begin{theorem}[Proposition~\ref{p:flag}] If $X = H/P$ is a generalized flag variety, then \begin{equation}\label{e:con}\overline{A(L)} = \op{NE}(L)^\vee.\end{equation}
\end{theorem}  The closure is necessary here, as the equivalence class $A(L)$ may contain points of its boundary.  We expect that \eqref{e:con} holds  more generally.

Finally, we consider the slightly different setting of a quotient of a normal affine variety $V$.  In this case, $\op{Pic}^G(V) = \chi(G)$.  We prove: \begin{theorem}[Theorem~\ref{vect}]\label{t:1.7} If $V^{ss}(\theta)$ is nonempty, then 
$$A(\theta) = \op{relint}\left(\op{NE}(\theta)^\vee\right).$$
\end{theorem}

\begin{example}
Let us illustrate the theorem in a simple example.  Let $T = \GM$ act on $V = \af^1$ by scaling, and let $\theta = \op{id} \in \chi(T)$ be the identity character.  The GIT quotient $V\sslash_\theta T$ is  a point.  In this case the cone $A(\theta)$ consists of positive rational multiples of $\theta$.  After identifying $\chi(T) \otimes \QQ$ with $\QQ$, $A(\theta)$ is the cone $\QQ_{>0}$.

For any $d \in \ZZ_{\geq 0}$, consider the $\theta$-quasimap  $f: \PP^1 \to [V/T]$ given in coordinates by
$$ [s:t] \mapsto [s^d].
$$
  It follows that the cone $\op{NE}(\theta)$ contains the ray $$\QQ_{\geq 0} \subset \QQ \cong \Hom(\chi(T), \QQ).$$  From here one can check that $\op{NE}(\theta)$ in fact equals $\QQ_{\geq 0}$. We conclude that 
$$A(\theta) = \op{relint}\left(\op{NE}(\theta)^\vee\right),$$
in agreement with Theorem~\ref{t:1.7}.
%quasimap $f: C \to [V/T]$ is determines a line bundle $L = f^*([\cc O_V(\theta)/T]) \to C$ and a section $s = f^*(\op{taut})$.  Here $[\cc O_V(\theta)/T]$ denotes the line bundle $\pi_1: [(V \times V)/T] \to [V/T]$ over the stack quotient and
% $\op{taut}$ is the tautological section induced by the diagonal map $\Delta: V \to V \times V$.  The quasimap $f$ is $\theta$-effective if and only if the section $s$ is nonvanishing on each component of $C$.  This is possible only if the degree of $L$ is nonnegative on each component, therefore $\op{NE}(\theta)$ is contained in the ray $\QQ_{\geq 0} \subset \QQ \cong \Hom(\chi(T), \QQ)$.
\end{example}
%
%Let $\op{NS}^G(X) = \op{Pic}^G(X)/\op{Pic}(X)_0$ denote the group of $G$-linearized line bundles on $X$ up to algebraic equivalence.  Let $\op{Amp}^G(X) \subset \op{NS}^G(X)$ be the cone of  $G$-linearized ample line bundles.  
%For each line bundle 

%
%
%Let $\Amp(X)$ denote the ample cone in $N^1(X)_\RR$
%THIS SHOULD BE THE NERON SEVERI GROUP
%, let $\Nef(X)$ denote the nef cone of numerically effective divisor classes.  Let $\NE(X)$ denote the cone of curve classes spanned by effective curve classes in $N_1(X)_\RR$.  
%
%One statement of Kleiman's criterion is that $\Amp(X)$ is the interior of $\Nef(X)$ which, by definition, is $\NE(X)^\vee$.  In particular, a line bundle $L \to X$ is ample if and only if $$L \cdot [C] > 0$$
%for all $[C] \in \NE(X)$.
%
%We will give an analogous statement for linearized stack quotients.
%
%\note{Does your notion of ample correspond to ample in the sense of the intro here: https://arxiv.org/pdf/2111.01068.pdf ? }
% (Has anyone computed the cohomology of twisted Grassmannians?)

%We are not sure if such a statement will prove useful in practice.  The main point is philosophical, that one can in theory test the wall and chamber structure of a GIT stack quotient by replacing curves with quasimaps.

%
%In the setting of stack quotients, Ciocan-Fontanine--Kim introduced a notion of Effective quasimaps, which are an invariant of an "L-polarized stack quotient" <-- make sense of this

\subsection{Acknowledgements}

I am grateful to Andres Fernandez Herrero and Victoria Hoskins for helpful correspondences and conversations, and to Jeff Achter and the anonymous referee for their valuable comments and suggestions on earlier drafts. This work was partially supported by NSF grant DMS-1708104 and Simons Foundation Travel Grant 958189.

This paper is dedicated to the memory of Prof. Bumsig Kim.  Through the beauty of his mathematics and the generosity with which he shared both his time and ideas, he had a great influence on me and so many others.  
%I always came away from conversations with him feeling 
I am grateful for our time together.
\section{GIT setup}

Fix $G$  a reductive group acting on a normal projective algebraic variety $X$, both defined over an algebraically closed field $k$.  %Let $\cX denote the stack quotient $[X/G]$.
%A line bundle $\ccL \to \cX$ is equivalent to 
Let $\pi: L \to X$ be a $G$-linearized line bundle, i.e. a line bundle with a $G$-action for which $\pi$ is $G$-equivariant.
%We will call a choice of $\cL$ a \emph{linearization} of $\cX$. 
%Given a linearization $\cL \to \cX$, 
We recall the definition of semi-stable and stable points of $X$ with respect to $L$.
\begin{definition}\label{d:ss}
A geometric point $x$ in $X$ is $L$-semi-stable if, for some $m \geq 0$, there exists a section $\sigma \in \Gamma(X, L^{\otimes m})^G$ such that $\sigma(x) \neq 0.$ We denote the associated open subset of $X$ by $X^{ss}(L)$.
%
%The locus of $L$-semi-stable points is given by
%$$X^{ss}(L) := \{ x\in X| \text{there exists }m \geq 0,  \sigma \in \Gamma(X, L^{\otimes m})^G \text{ such that } \sigma(x) \neq 0 \}.$$
%The locus of $L$-stable points is
A geometric point $x$ in $X$ is $L$-stable if $x$ is $L$-semi-stable, the orbit $G \cdot x$ is closed in $X^{ss}(L)$, and $|G_x| < \infty$.  We denote the associated open subset by $X^{s}(L)$.
%$$ X^{s}(L) := \{ x \in X^{ss}(L) | G \cdot x \text{ is closed in } X^{ss} \text{ and } |G_x| < \infty\}.$$
The $L$-unstable locus is $X^{us}(L):= X \setminus X^{ss}(L)$.  When $L$ is fixed we will sometimes refer to these as simply the semi-stable, stable, and unstable loci respectively.
\end{definition}

\begin{definition} Define the \emph{GIT quotient} of $X$ (with respect to $L$) to be:
$$X \sslash_L G := \op{Proj}\left( \bigoplus_{r \geq 0} H^0(X, L^{\otimes r})^G  \right).$$
Let $[X/G]$ denote the stack quotient of $X$ by $G$. Define the \emph{GIT stack quotient} to be
$$[X \sslash_L G] := [X^{ss}(L) / G].$$
\end{definition}
%We have the following diagram:
%\[
%\begin{tikzcd}
%& \left[X \sslash_L G\right]  \ar[dl, twoheadrightarrow] \ar[dr, hook] & \\
%X \sslash_L G & & \left[X/G\right] 
%\end{tikzcd}
%\]
The map $X^{ss}(L) \to X \sslash_L G$ is a categorical quotient, and consequently \\ $X \sslash_L G$ depends only on $X^{ss}(L)$ and not $L$ itself.  This motivates the following definition.

\begin{definition}
Two $G$-linearized line bundles $L$ and $L'$ on $X$ are said to be \emph{GIT equivalent} if
$$X^{ss}(L) = X^{ss}(L').$$
\end{definition}

By the work of Dolgachev--Hu \cite{DH}, there are only a finite number of GIT equivalence classes within the ample cone.

\subsection{A test for stability}\label{HMtest}
%Let $X$ be a projective variety acted on by a reductive group $G$, let $L$ be a $G$-linearized line bundle on $X$.  	$\lambda: \GM \hookrightarrow G$
A \emph{one parameter subgroup (1-PS) of $G$} is a non-trivial group homomorphism $\lambda: \GM \hookrightarrow G$.  Given a 1-PS $\lambda$ and a point $x \in X$, by properness of $X$ the map 
\begin{align*}
\lambda_x: \GM &\to X \\
t &\mapsto \lambda(t) \cdot x
\end{align*}
%$\lambda_x: \GM \to X$ defined by
%$$t \mapsto \lambda(t) \cdot x$$
extends uniquely to a morphism
$\bar \lambda_x: \af^1 \to X.$
Let $x_0 = \bar \lambda_x(0).$
By construction, $x_0$ is fixed by the action of $\GM$, so $\GM$ acts on $L|_{x_0}$ with some weight\footnote{There are different sign conventions for $\rho^L_{(\lambda, x)}$ based on whether one defines the line bundle associated to an invertible sheaf $\cc L$ to be $\spec\left(\Sym^\bullet \cc L\right)$ or $\spec\left(\Sym^\bullet \cc L^\vee\right)$.  We take the latter convention.}:%\note{check this}, denoted  $\rho^L_{(\lambda, x)}$:
$$\lambda(t) \cdot v = t^{\rho^L_{(\lambda, x)}} v$$ for all $v \in L|_{x_0}$ and  $t \in \GM.$

%In other words, the map $i$ extends to a map 
%$\bar \lambda: \CC \to X$.
The celebrated Hilbert--Mumford criterion asserts that one can test the (semi-)stability of a point $x \in X$ using the asymptotics of all 1-PS's.
\begin{theorem}[Hilbert--Mumford criterion \cite{GIT}]\label{t:HM}
If $L \to X$ is an \emph{ample} $G$-linearized line bundle, then
$$x \in X^{ss}(L) \iff \rho^L_{(\lambda, x)} \geq 0 \text{ for all 1-PS's } \lambda: \GM \hookrightarrow G;$$
$$x \in X^{s}(L) \iff \rho^L_{(\lambda, x)} > 0 \text{ for all 1-PS's } \lambda: \GM \hookrightarrow G.$$
\end{theorem}

\subsection{Variation of GIT}\label{s:vGIT}

\begin{notation}
Let $W$ be a finite-dimensional vector space over $\QQ$, and let $$W^\vee = \Hom(W, \QQ)$$ denote the dual vector space.  For $C \subset W$ a convex cone, define the dual cone $C^\vee \subset W^\vee$ by
$$C^\vee = \{ f \in W^\vee | f(c) \geq 0 \text{ for all } c \in C\}.$$
Denote by $\op{relint}(C)$ the relative interior of $C$.
\end{notation}

Let $\op{Pic}^G(X)$ denote the Picard group of $G$-linearized isomorphism classes of line bundles on $X$.  We say $L_1, L_2 \in \op{Pic}^G(X)$ are \emph{$G$-algebraically equivalent} \cite{Thad}  if there is a connected variety $T$,  points $t_1, t_2 \in T$, and a $G$-linearized line bundle $L \to T \times X$ (where the action of $G$ on $T \times X$ is induced by the action on $X$), such that 
\[
L|_{t_1 \times X} \cong L_1, \hspace{.5 cm} L|_{t_2 \times X} \cong L_2.
\]
Define $\op{NS}^G(X)$ to be the set of $G$-algebraic equivalence classes in $\op{Pic}^G(X)$.  By \cite[Proposition~2.1]{Thad}, after tensoring with $\QQ$ we have an exact sequence
$$0 \to \chi(G)_\QQ \to \op{NS}^G(X)_\QQ \to \op{NS}(X)_\QQ \to 0,$$
where $\chi(G)$ is the group of characters of $G$ and the last map forgets the $G$-linearization.
%
%Let $$c_1: \op{Pic}^G(X) \to H^2(X; \ZZ)$$ denote the map which forgets the linearization and then takes the first Chern class.  The kernel of $c_1$ is canonically isomorphic to $\chi(G) \times \op{Pic}(X)_0$, where $\chi(G) = \Hom(G, \GM)$.  We use this isomorphism to identify $\op{Pic}(X)_0$ with the subgroup $\{1\} \times \op{Pic}(X)_0 \subset \ker (c_1).$ 
%Alternatively, we may view $\op{Pic}(X)_0$ as the kernel of the equivariant first Chern class
%$$c_1^G: \op{Pic}^G(X) \to H^2_G(X; \ZZ).$$ A line bundle $L \in \op{Pic}^G(X)$ is called \emph{algebraicly trivial} if it lies in $\op{Pic}(X)_0$.  
%
%Define $$\op{NS}^G(X) := \op{Pic}^G(X) /\op{Pic}(X)_0.$$  
Consequently, $\op{NS}^G(X)_\QQ$ is a finitely generated abelian group.  By the Hilbert--Mumford criterion, if two ample $G$-linearized line bundles $L$ and $L'$ define the same element in $\op{NS}^G(X)$, then they are GIT equivalent.
%, i.e. $X^{ss}(L) = X^{ss}(L ')$.  

Denote by $\op{Amp}^G(X) \subset \op{NS}^G(X)_\QQ$ the convex cone spanned by the classes of $G$-linearized ample line bundles.  Denote by $C^G(X) \subset \op{Amp}^G(X)$ the cone of $G$-effective ample $G$-linearized  line bundles, i.e. those $[L] \in \op{Amp}^G(X)$ for which $X^{ss}(L)$ is nonempty.  Such classes are called \emph{$G$-ample}.
\begin{definition}\label{cones}
For $[L]$ a class in  $C^G(X)$, define
$$C(L) = \{[L'] \in C^G(X) | X^{ss}(L) \subset X^{ss}(L')\}.$$
Let $C^\circ(L)$ denote the cone of $G$-ample  line bundles which are GIT equivalent to $L$:
$$C^\circ(L) = \{[L'] \in C^G(X) | X^{ss}(L) = X^{ss}(L')\}.$$
\end{definition}
Following the work of Thaddeus \cite{Thad} and Dolgachev--Hu \cite{DH}, Ressayre proved the following:
\begin{theorem}[\cite{Res}]\label{Res}
%For $l_0 \in C^G(X)$, the set $C(l_0)$ is a closed convex rational polyhedral cone in $C^G(X)$.  
The sets  $C(L)$ are closed convex rational polyhedral cones in $\op{Amp}^G(X)$ which form a fan covering $C^G(X)$.  The GIT equivalence class $C^\circ(L)$ is the relative interior $\op{relint}(C(L))$.
\end{theorem}
%In particular, the theorem shows that $X^{ss}(L) \subset X^{ss}(l)$ if and only if $l$ is in the closure of the set of $l' \in C^G(X)$ which are GIT equivalent to $L$.
%
\section{Quasimaps}\label{s:q}
Fix a $G$-ample line bundle $L \to X$ as above.
\begin{definition} An $L$-quasimap 
$f: C \to [X/G]$
is a morphism from a smooth curve
 $C$ to the stack $[X/G]$
for which $f^{-1}([X^{ss}(L)/G])$ is a dense open subset.
\end{definition}

A $G$-linearized line bundle $N \to X$ induces a line bundle on the stack quotient $[X/G]$ which we will denote by $[N/G] \to [X/G]$.
Given a map $$f:  C \to [X/G]$$ one can pull back
$[N/G]$ to $C$ and take degree  to obtain an integer $$\deg(f^*[N/G]).$$  
\begin{definition} We define the degree of $f$ to be the element of 
$\Hom(\op{NS}^G(X), \ZZ)$
given by 
$$N \mapsto \deg(f^*[N/G]).$$
Those classes 
 $\beta \in \Hom(\op{NS}^G(X), \ZZ)$ which are realized as the degree of an $L$-quasimap will be called $L$-quasimap classes.
 Denote by $$\op{NE}(L) \subset \Hom(\op{NS}^G(X), \QQ)$$ the cone generated by $L$-quasimap classes.
\end{definition}
\begin{proposition}\label{p:posdeg}
Given an $L$-quasimap $f: C \to [X/G]$, the degree $\deg(f^*[L/G])$ is non-negative.\end{proposition}

\begin{proof}
This was proven in \cite[Lemma~3.2.1]{CKM} in the case that $X$ is an affine variety over $\CC$.  We recall the argument here with the necessary modifications.

%If $C$ has orbifold structure, take a finite cover $\hat C \to C$ and consider the composition $\hat f: \hat C \to [X/G]$.  Then  $\deg(\hat f^*[L/G])$ is a positive multiple of $\deg(f^*[L/G])$.  It therefore suffices to consider the case when $C$ has no orbifold points.  

Assume without loss of generality that $C$ is irreducible.  Choose a point $c \in C$ such that $f(c)$ lies in $[X^{ss}(L)/G]$.  Such a point exists because $f$ is assumed to be an $L$-quasimap.  Let $x \in X^{ss}(L)$ be a point mapping to $f(c)$ and let $\sigma \in \Gamma(X, L^{\otimes m})^G$ be a  $G$-invariant section of a positive  power of $L$ such that $\sigma(x) \neq 0$.  The section $\sigma$ descends to a section 
$$\bar \sigma \in \Gamma([X/G], [L^{\otimes m}/G])$$ which is nonzero at $f(c)$.   The degree  $\deg(f^*[L^{\otimes m}/G])$ is therefore non-negative because it has a non-vanishing section $f^*(\bar \sigma)$.  We conclude by noting that $\deg(f^*[L/G]) = \frac{1}{m} \deg(f^*[L^{\otimes m}/G]).$
\end{proof}
\begin{remark}\label{rgen}
Proposition~\ref{p:posdeg} holds with the same proof for $X$ a quasiprojective variety with an action of $G$.
\end{remark}
%The line bundle $L$ descends to an ample line bundle $\cc O_{X\sslash_L G}(1)$ on the (non-stacky) GIT quotient $X\sslash_L G$.  Choose a multiple $m$ for which $\cc O_{X\sslash_L G}(m)$ is very ample.  We have the following commutative diagram:
%\[
%\begin{tikzcd}
%C \ar[r, "f"] \ar[dr, dashrightarrow]& \left[X/G\right] \ar[r, "\psi"] & \left[ (\Gamma(X, L^{\otimes m})^G)^\vee / \GM \right] \\
%%& \left[X \sslash_L G\right] \ar[u, hookrightarrow] \ar[d] & \\
%& X \sslash_L G \ar[r, hookrightarrow]  & \PP \left((\Gamma(X, L^{\otimes m})^G)^\vee \right)\ar[u, hookrightarrow] 
%\end{tikzcd}
%\]
%
%The line bundle $\cc O_{[\af^r/\GM]}(1)$ on $\left[ (\Gamma(X, L^{\otimes m})^G)^\vee / \GM \right]$ defined by the weight 1 character of $\GM$ pulls back to $f^*[L/G]$ on $C$ and to $\cc O_{X\sslash_L G}(m)$ on $X \sslash_L G$.  Choose a point $c \in C$ for which the map to 
%%By assumption there is an open subset $U$ which maps to
% $X \sslash_L G$ is defined.  Choose a section $s \in \Gamma(X \sslash_L G, \cc O_{X\sslash_L G}(m))$ which does not vanish on the image of $c$.  This section is the pullback of a section $ \tau \in \Gamma([\af^r/\GM], \cc O_{[\af^r/\GM]}(1))$.  Then $(\psi \circ f)^*(\tau)$ is a nonzero section of $f^*[L/G]$.  Therefore $\deg(f^*[L/G]) \geq 0$.
% 
% Furthermore, if $\deg(f^*[L/G]) = 0$ then $(\psi \circ f)^*(\tau)$ is a constant section, in which case $f$ is a constant map.

%As an immediate corollary, we obtain:
\begin{corollary}\label{c:1}
We have the following inclusion of cones:
$$C(L) \subset \op{NE}(L)^\vee.$$
\end{corollary}
\begin{proof}
If $[N] \in C(L)$, then $X^{ss}(L) \subset X^{ss}(N)$, thus
%If a line bundle $N$ is GIT equivalent to $L$, then 
any $L$-quasimap $f: C \to [X/G]$ is also an $N$-quasimap.  By the proposition, we see that $\deg(f^*[N/G]) $ will then be non-negative.
\end{proof}

\section{Hilbert--Mumford criterion and quasimaps}
\subsection{The main construction}\label{s:mc}
In this section we relate the weights $\rho^L_{(\lambda,x)}$ to the degrees of certain quasimaps.  %The main construction is as follows.

%\begin{lemma}\label{GX}
%If $x \in X^{ss}(L)$, then $g\cdot x \in X^{ss}(L)$ for all $g \in G$.
%\end{lemma}

\begin{construction}\label{const}
Fix  $\lambda: \GM \hookrightarrow G$  a 1-PS and  $x\in X^{ss}(L) $ a semi-stable point.  Following Section~\ref{HMtest}, the morphism $\lambda_x: \GM \to X$ given by $t \mapsto \lambda(t) \cdot x$
extends to a morphism 
$$\bar \lambda_x: \af^1 \to X.$$

Define the map $\hat \phi_{(\lambda, x)}: \af^2 \setminus \{0\} \to X$ by 
$$\hat \phi_{(\lambda, x)} (s,t) \mapsto \bar \lambda_x(t).$$
Note that when $t \neq 0$, $(s,t)$ maps to $X^{ss}(L)$.  If we let $\GM$ act on $\af^2 \setminus \{0\}$ by scaling and on $X$ via $\lambda$, then $\hat \phi_{(\lambda, x)}$ is $\GM$-equivariant.  It therefore induces an $L$-quasimap
$$\phi_{(\lambda, x)}: \PP^1 =  [\af^2 \setminus \{0\}/ \GM] \to [X/G].$$
\end{construction}

\begin{lemma}
Given a $G$-linearized line bundle $N \to X$, the degree of $\phi_{(\lambda, x)}^*([N/G])$ is $\rho^N_{(\lambda, x)}$.
\end{lemma}

\begin{proof}
The morphism $\phi_{(\lambda, x)}$ factors through the projection onto the second factor:
$$\phi_{(\lambda, x)}: [\af^2 \setminus \{0\}/\GM] \xrightarrow{\pi_2} [\af^1/\GM] \xrightarrow{[\bar \lambda_x/{\GM}]} [X/G]$$
The pullback of $N$ to $\af^1$ via $\bar \lambda_x$ is trivial, so the isomorphism class of  \\ $[\bar \lambda_x/{\GM}]^*([N/G])$ is determined by the weight of the action of $\GM$ on $\bar \lambda_x^*(N)|_{0}$, which is $\rho^N_{(\lambda, x)}$.
\end{proof}

\subsection{A Kleiman criterion, and some questions}
Construction~\ref{const} gives a partial converse to Corollary~\ref{c:1}.
\begin{proposition}\label{p:2}
If $L$ and $N$ are ample $G$-linearized line bundles such that $[N]$ is not contained in ${C(L)}$ then there exists an $L$-quasimap $f: C \to [X/G]$ such that $\deg(f^*[N/G]) < 0$.  In other words,
$$\op{Amp}^G(X) \setminus {C(L)} \subset ({\op{NE}(L)^\vee})^c.$$
\end{proposition}

\begin{proof}
If $[N]$ is not contained in ${C(L)}$, then there exists a point $$x \in X^{ss}(L) \setminus X^{ss}(N).$$  By the Hilbert--Mumford criterion, there must exist a 1-PS $\lambda: \GM \hookrightarrow G$ for which $\rho^N_{(\lambda, x)} < 0$.   The associated quasimap $\phi_{(\lambda, x)}$ then has the desired properties.
\end{proof}
Combining Theorem~\ref{Res}, Corollary~\ref{c:1}, and Proposition~\ref{p:2}, we obtain the following analogue of Kleiman's criterion.
\begin{theorem}\label{maint} 
Fix a $G$-linearized normal projective variety $(X, G, L)$, with $L \to X$ a $G$-ample line bundle.  Then
\begin{equation}\label{e:mt0}C^\circ(L) = \op{relint}\left(\op{NE}(L)^\vee\right)\cap \op{Amp}^G(X).\end{equation}
In other words, a $G$-linearized ample line bundle $L' \to X$ is GIT equivalent to $L$ if and only if 
%\note{check this!! relint of cones is confusing!}
$$L \cdot \gamma \text{ is }
\left\{\begin{array}{ll}
 > 0 &\text{for all } \gamma \in \overline{\op{NE}(L)} \text{ such that } - \gamma \notin \overline{\op{NE}(L)}
 \\
 =0 &\text{for all } \gamma \in \overline{\op{NE}(L)}  \text{ such that }  - \gamma \in \overline{\op{NE}(L)}
\end{array}\right\}.
$$

\end{theorem}
%\note{do we need to assume $X$ is smooth like in DH?}
Assume $\op{Pic}(X)_0$ is torsion, so $\op{NS}^G(X)_\QQ = \op{Pic}^G(X)_\QQ$.  This guarantees that GIT equivalence classes are well-defined on all of $\op{Pic}^G(X)_\QQ$ and not just on the ample cone $\op{Amp}^G(X)$.
Under this assumption, let 
$$A(L) := \{L' \in \op{Pic}^G(X)_\QQ | X^{ss}(L') = X^{ss}(L)\}$$
denote the GIT equivalence class of $L$.
If $L$ is $G$-ample, then 
$$A(L) \cap \op{Amp}^G(X) = C^\circ(L).$$
%where $\op{relint}$ denotes the relative interior of the cone.
%Wit this we obtain the following corollary.
In this case, Theorem~\ref{maint} may be rewritten as:
%\begin{corollary}\label{mainc}  
% If $\op{Pic}(X)_0$ is torsion and $L$ is $G$-ample, then
\begin{equation}\label{mt} A(L) \cap \op{Amp}^G(X)  =  \op{relint}\left(\op{NE}(L)^\vee\right)\cap \op{Amp}^G(X).\end{equation}
%\end{corollary}

Theorem~\ref{maint} and equation \eqref{mt} may be viewed as an analog of Kleiman's criterion on the ample cone.  Here $A(L)$ plays the role of the ample cone with respect to the linearization $L$, and $\op{NE}(L)$ plays the role of the  cone of curves.
It would be nice to understand when this relationship between cones extends beyond $\op{Amp}^G(X)$.
\begin{question}\label{qq1}
Assume $\op{Pic}(X)_0$ is torsion.  For which triples $(X, G, L)$ does
\begin{equation}\label{e:hope}A(L) = \op{relint}\left(\op{NE}(L)^\vee\right)?\end{equation}
\end{question}
It is not hard to construct examples where $A(L)$ contains some points of its boundary.   (Consider, for instance, the action of $G = \GM$ on \\ $\PP^1 \times \PP^1$ which scales one of the $\PP^1$ factors.)
We are hopeful, however, that the following weaker condition is more common: 
%$\overline{A(L)} = \op{NE}(L)^\vee.$
\begin{question}\label{qq2}
 Find conditions on $X$, $G$, and $L$, such that
\begin{equation}\label{e:hope2}\overline{A(L)} = \op{NE}(L)^\vee.\end{equation}
\end{question}
We remark in passing that if $L$ is not ample, Question~\ref{qq2} is already interesting when the group $G$ is trivial.

One case where \eqref{e:hope2} holds is when $X$ is a generalized flag variety.
\begin{proposition}\label{p:flag}
Suppose $X = H/P$ where $P$ is a parabolic subgroup of a reductive group $H$.  If $L \to X$ is a $G$-ample line bundle then $$\overline{A(L)} = \op{NE}(L)^\vee.$$
\end{proposition}
\begin{proof}
The proof of Corollary~\ref{c:1} shows that $\overline{A(L)} \subset \op{NE}(L)^\vee$.  Thus
by \eqref{mt}, it suffices to show $\op{NE}(L)^\vee \subset \overline{\op{Amp}^G(X)}$.
Suppose $$N \in \op{Pic}^G(X)_\QQ \setminus \overline{\op{Amp}^G(X)}.$$  We will construct an $L$-quasimap $f$ such that $\deg(f^*[N/G]) < 0$.  

By Kleiman's criterion (Theorem~\ref{tK}), there exists an irreducible curve $C$ with $N \cdot [C] < 0$.  We can move $C$ using the action of $H$ on $X$ so that it intersects the $L$-semi-stable locus.  In particular, there exists an element $h \in H$ and a curve $C'$ which intersects $X^{ss}(L)$ such that $C' = h\cdot C$.  Let $\hat C \to C'$ be the normalization, and define $f: \hat C \to [X/G]$ to be the composition $$\hat C \to C' \hookrightarrow X \to [X/G].$$
Then $\deg(f^*[N/G]) $ is a positive multiple of $N \cdot [C'] $ which is negative.  By construction of $C'$, $f$ is an $L$-quasimap. 
\end{proof}

We conclude this section by mentioning another  line of inquiry which might be of interest.
The definition of the semi-stable locus (Definition~\ref{d:ss})
%The definition of $L$-semi-stable points
 naturally extends to the setting of $\cc X$ an Artin stack with a fixed line bundle $\cc L \to \cc X$ \cite{Alp}.  On the other hand, the $L$-quasimaps $\phi_{(\lambda, x)}$ of Construction~\ref{const} are closely related to $\Theta$-stability for Artin stacks as defined in \cite{HL} and $\cc L$-stability of \cite{Hein}.  Questions~\ref{qq1} and~\ref{qq2} above can therefore be extended to Artin stacks, and Theorem~\ref{maint} may be viewed as concerning the special case that $\cc X = [X/G]$ is the stack quotient  of a projective variety by a reductive group.
\begin{question}
Let $\cc X$ be an Artin stack and let $\cc L \to \cc X$ be a line bundle.  Find conditions on $\cc X$ and $\cc L$ such that
$$\overline{A(\cc L)} = \op{NE}(\cc L)^\vee.$$
\end{question}

\subsection{The case of a fixed ample line bundle}
For general $X$, the result is cleaner if we restrict our attention to powers of a fixed ample line bundle $P \to X$ and allow the linearization to vary.

%Assume the action of $G$ on $X$ is nontrivial.
Let $ P \to X$ be an ample line bundle \emph{without} a choice of $G$-linearization.  
\begin{definition}
Let $\op{NS}^G_{P}(X)_\QQ$ denote the subspace of $\op{NS}^G(X)_\QQ$ spanned by all $G$-linearizations of powers of $ {P}$.  
Define
\begin{align*}
A_P(L) &:=  C^\circ(L) \cap \op{NS}^G_{P}(X)_\QQ; \\
C_P^G(X) &:= C^G(X) \cap \op{NS}^G_{P}(X)_\QQ.
\end{align*}
Elements of $\Hom(\op{NS}^G(X), \QQ)$ naturally restrict to $\op{NS}^G_{P}(X)$. %elements of $\Hom(\op{NS}^G_{P}(X), \QQ)$.  
Let $\op{NE}_P(L)$ denote the image of $\op{NE}(L)$ in $\Hom(\op{NS}^G_{P}(X), \QQ)$. 
\end{definition}

\begin{theorem}\label{fixedt}
In $\op{NS}^G_{P}(X)_\QQ$, if $L \in C_P^G(X)$, then $$A_P(L) = \op{relint}\left(\op{NE}_P(L)^\vee\right).$$
\end{theorem}
\begin{proof}
By intersecting \eqref{e:mt0} with $\op{NS}^G_{P}(X)_\QQ$, we have that 
$$A_P(L)  = \op{relint}\left(\op{NE}_P(L)^\vee\right) \cap \op{Amp}^G(X) .$$

Thus, it suffices to show that if $L \in C^G_P(X)$ and $$0 \neq N \in  \op{NS}^G_{P}(X)_\QQ \setminus \op{Amp}^G(X),$$ then there exists an $L$-quasimap $f$ for which $\deg(f^*[N/G]) < 0$.  

If $N$ is not ample then after forgetting the linearization it is a non-positive power of $P$ and therefore has non-positive degree.  If $\deg(N) <0$, let \\$X \hookrightarrow \PP^r$ be the projective embedding of $X$ determined by a very ample power of $P$.  Intersecting $X$ with a general linear subspace of the correct dimension, we obtain a curve $C \subset X$ which intersects the semi-stable locus $X^{ss}(L)$.  If $C$ is reducible, choose an irreducible component $C'$ which intersects the semi-stable locus $X^{ss}(L)$.
If $C'$ is not smooth, let $\hat C \to C'$ be a resolution.  Let $f: \hat C \to [X/G]$ denote the composition
%$$ \hat C \to C' \hookrightarrow X \to [X/G].$$
\[
\begin{tikzcd}
 {\hat C} \ar[r] %\ar[rrr, bend left = 15, "f"] 
 &C' \ar[r, hookrightarrow] &X \ar[r] &\left[X/G\right].
\end{tikzcd}
\]
As in the proof of Proposition~\ref{p:flag}, $f$ is an $L$-quasimap and 
 $\deg(f^*[N/G]) <0.$
%  is a positive multiple of $\deg(N|_{C'})$ which is negative.  By construction of $C'$, $f$ is an $L$-quasimap. 

If $N$ has degree zero then it is $\cc O_{X}(\theta)$ for some non-trivial character $\theta \in \chi(G)$.  Choose a 1-parameter subgroup $\lambda: \GM \hookrightarrow G$ such that $\theta \circ \lambda $ has weight ${-d}$ for some $d>0$.  Choose $x \in X^{ss}(L)$ and consider $ \phi_{(\lambda, x)}$ as in Construction~\ref{const}.  Then $\phi_{(\lambda, x)}$ is an $L$-quasimap, and $\deg(\phi_{(\lambda, x)}^*[N/G]) = -d <0$.
\end{proof}
This immediately implies the following special case.
\begin{corollary}
If  $\op{NS}(X)_\QQ \cong \QQ$, and  $L \in C^G(X)$, then $$C^\circ(L) = \op{relint}\left(\op{NE}(L)^\vee\right).$$
\end{corollary}
%\note{Picard or NS? E or C??}
%\begin{proof}
%By Theorem~\ref{maint}, it suffices to show that if $L \in \op{Amp}^G(X)$ and $0 \neq N \notin \op{Amp}^G(X)$, then there exists an $L$-quasimap $f: C \to [\PP^n/G]$ for which $\deg(f^*[N/G]) < 0$.  If $N$ is not ample than it has non-positive degree.  If $\deg(N) <0$, then choose a line $\PP^1 \subset \PP^n$ which intersects the semi-stable locus $X^{ss}(L)$.
%Let $f: \PP^1 \to [\PP^n/G]$ denote the composition of the inclusion of the line $\PP^1 \hookrightarrow \PP^n$ with the quotient map $\PP^n \to [\PP^n/G]$.  If $N$ is degree zero then it is $\cc O_{\PP^n}(\theta)$ for some non-trivial character $\theta: G \to \GM$.  Choose a 1-parameter subgroup \phi_{(\lambda, x)}such that $\theta \circ \lambda(t) = \lambda^{-d}$ for some $d>0$.  Choose $x \in X^{ss}(L)$ and let $f = \phi_{(\lambda, x)}$ from the previous section.  Then $\phi_{(\lambda, x)}$ is an $L$-quasimap, and $\deg(f^*[N/G]) = -d <0$.
%\end{proof}

%\subsection{Some questions}

%
%We make the following optimistic conjecture.
%\begin{conjecture}\label{conj}
%Then
%$$A(L) = \op{relint}\left(\op{NE}(L)^\vee\right).$$
%\end{conjecture}
%\note{As written this seems false... maybe the best you can hope for is $\overline {A(L)} = \op{NE}(L)^\vee$
%ANALIZE SOME TRIVIAL CASES OF PRODUCT OF 2 PROJECTIVE SPACES TO CHECK THIS.}
%\subsection{Speculation}

%Let $\cc X$ denote the Artin stack $[X/G]$ and let $\cc L = [L/G]$.  
%
%Questions~\ref{qq1} and~\ref{qq2} have analogues for linearized Artin stacks

\section{Quotients of affine varieties}\label{s:toric}
In this section we extend Theorem~\ref{maint}
to the affine setting.  Let $G$ be a reductive group acting linearly on a normal affine variety $V$.  In this case %\cite{Ha}  
$ \op{Pic}^G(V) = \chi(G)$.
With this identification we will use $\theta$ to denote both a character of $G$ as well as the associated $G$-linearized line bundle on $V$.

Given $\theta \in \chi(G)$, define a $\theta$-quasimap quasimap  $$f: C \to [V/G]$$ 
exactly as in Section~\ref{s:q}. The degree of a morphism $f: C \to [V/G]$ lies in $\Hom(\chi(G), \ZZ)$.  
Let $\op{NE}(\theta)$ denote the cone of degrees of $\theta$-quasimaps to $[V/G]$.

Following Definition~\ref{cones}, for $\theta$ a $G$-effective character in $\chi(G)_\QQ$,
let $$C_V(\theta) = \{\theta' \in \chi(G)_\QQ | V^{ss}(\theta) \subset V^{ss}(\theta')\}.$$ Let $A(\theta)$ denote the GIT equivalence class for a character $\theta \in \chi(G)$.  

To follow the arguments of the previous section, we require appropriate modifications of Ressayre's theorem (Theorem~\ref{Res}) and the Hilbert--Mumford criterion in the case of a group acting on an affine variety.  The generalization of Theorem~\ref{Res} to the affine setting was given by Halic \cite{Ha}.
\begin{theorem}[\cite{Ha}]\label{c:R2}
The sets  $C_V(\theta)$ are closed convex rational polyhedral cones which form a fan in $\chi(G)_\QQ$.  The GIT equivalence class $A(\theta)$ is the relative interior of $C_V(\theta)$.
\end{theorem}

Next we state the analogue of the Hilbert--Mumford criterion for affine varieties, which was obtained by King in \cite{King}.  As with the Hilbert--Mumford criterion for projective varieties (Theorem~\ref{t:HM}), this is a consequence of the fact that when $G$ is a reductive group stability can be tested using 1-PS's.  In this case, for a 1-PS $\lambda: \GM \hookrightarrow G$, there is no guarantee that the map $\lambda_v: \GM \to V$ given by $t \mapsto \lambda(t) \cdot v$ extends to a morphism $\bar \lambda_v: \af^1 \to V$.  It turns out that one can effectively ignore those 1-PS's for which $\lambda_v$ does not extend when testing for stability.  When $\bar \lambda_v: \af^1 \to V$ exists, the weight $\rho^\theta_{(\lambda, v)}$ is defined as in Section~\ref{HMtest}.
\begin{theorem}[\cite{King, GHH}]\label{l:HMa}
A point $v \in V$ is $\theta$-semi-stable if and only if 
$\rho^\theta_{(\lambda, v)} \geq 0$ for all  1-PS's $\lambda: \GM \hookrightarrow G$ such that $\lambda_v: \GM \to V$ extends to $\af^1$.
\end{theorem}
%and if the weight of $\theta$ is zero, then the weights of nonzero coordinates of $v$ are not all positive or all negative. <--- FALSE: 
 We conclude with the following analogue of Theorem~\ref{maint}.
\begin{theorem}\label{vect}
Let $G$ be a reductive group acting linearly on a normal affine variety $V$.  Then
for each $\theta \in \chi(G)$ with $V^{ss}(\theta) \neq \emptyset$, we have
$$A(\theta) = \op{relint}\left(\op{NE}(\theta)^\vee\right).$$
\end{theorem}

\begin{proof}
The argument follows the same steps as the proof of Theorem~\ref{maint}.
By \cite[Lemma~3.2.1]{CKM} (or Remark~\ref{rgen}), if a morphism $f: C \to [V/G]$ is a $\theta$-quasimap, then $\deg(f^*[\cc O_V(\theta)/G]) \geq 0$.  Therefore, 
$$C_V(\theta) \subset \op{NE}(\theta)^\vee.$$

Next, assume $\kappa \notin C_V(\theta)$.  Then there exists a point $ v \in V^{ss}(\theta) \setminus V^{ss}(\kappa)$.  By Theorem~\ref{l:HMa}, there exists a 1-PS $\lambda: \GM \hookrightarrow G$ such that
$\bar \lambda_v: \af^1 \to V$ exists and $\rho^\kappa_{(\lambda, v)} < 0$.
Using $\bar \lambda_v$, we can define $\phi_{(\lambda, v)}: \PP^1 \to [V/G]$ as in Construction~\ref{const}. 
Then $\phi_{(\lambda, v)}$ is a $\theta$-quasimap such that $\deg(\phi_{(\lambda, v)}^*[\cc O_V(\kappa)/G])<0$.  

Consequently, 
$$ C_V(\theta)^c \subset ({\op{NE}(\theta)^\vee})^c,$$
and therefore $$C_V(\theta) = {\op{NE}(\theta)^\vee}.$$
By Corollary~\ref{c:R2}, $A(\theta) = \op{relint}(C_V(\theta)).$ This proves the theorem. 
\end{proof}

%\section{Speculations}
%Danger: if L is not ample, then it seems to not be the case that GIT equivalence is preserved under the action of $Pic_0$ so not well defined in $\op{NS}^G$.  For example, let $G$ be trivial and $E$ an elliptic curve and let $L$ be trivial and $L'$ be a different degree 0 line bundle.  Then $E^{ss}(L) = E$ and $E^{ss}(L') $ is empty.
%\lambda: \GM \hookrightarrow G
%
%https://arxiv.org/pdf/1609.06058.pdf for another notion of HM criterion, and do Halpern-Leistner

\bibliographystyle{plain}
\bibliography{references2}

\end{document}